\documentclass[11pt]{amsart}

\usepackage[T1]{fontenc}
\usepackage[utf8]{inputenc}
\usepackage{lmodern}
\usepackage{microtype}
\usepackage{amsmath,amssymb,amsfonts,amsthm,mathtools}
\usepackage{enumitem}
\usepackage[dvipsnames]{xcolor}
\usepackage[colorlinks=true,linkcolor=blue,citecolor=blue,urlcolor=blue]{hyperref}

\numberwithin{equation}{section}

\newtheorem{theorem}{Theorem}[section]
\newtheorem{corollary}[theorem]{Corollary}
\newtheorem{proposition}[theorem]{Proposition}
\newtheorem{lemma}[theorem]{Lemma}
\newtheorem{definition}[theorem]{Definition}
\newtheorem{remark}[theorem]{Remark}
\newtheorem{example}[theorem]{Example}

\DeclareMathOperator{\Ric}{Ric}
\DeclareMathOperator{\diam}{diam}
\DeclareMathOperator{\supp}{supp}
\DeclareMathOperator{\Tr}{Tr}

\newcommand{\R}{\mathbb{R}}
\newcommand{\Ptwo}{\mathcal{P}_2}
\newcommand{\Prob}{\mathcal{P}}
\newcommand{\Wtwo}{W_2}
\newcommand{\dd}{\,\mathrm{d}}
\newcommand{\ip}[2]{\left\langle #1,#2\right\rangle}

\newcommand{\e}{\mathbf{e}}

\definecolor{revisionpurple}{RGB}{112,48,160}
\definecolor{REVISIONPURPLE}{RGB}{112,48,160}

\newcommand{\doi}[1]{\href{https://doi.org/#1}{doi:#1}}
\newcommand{\revcite}[1]{\textcolor{revisionpurple}{\NoHyper\cite{#1}\endNoHyper}}

\title[A Liouville theorem on Wasserstein space]
{A Liouville theorem for bounded empirically harmonic functions on $\mathcal{P}_2(M)$}

\author{Hongwei Yuan}
\address{Department of Mathematics, University of Macau}
\email{hwyuan@um.edu.mo}


\begin{document}

\begin{abstract}
The classical Liouville theorem states that every bounded harmonic function on Euclidean space is constant. On complete Riemannian manifolds, analogous conclusions hold under geometric assumptions such as nonnegative Ricci curvature. The quadratic Wasserstein space $\Ptwo(M)$ has no canonical infinite-dimensional Riemannian volume and hence no canonical Laplace--Beltrami operator. We introduce a natural finite-particle notion of harmonicity: a continuous function $u:\Ptwo(M)\to\R$ is called empirically harmonic if, for every $N\geq1$, its pullback under the empirical map
\[
\iota_N(x_1,\ldots,x_N)=\frac1N\sum_{i=1}^N\delta_{x_i}
\]
is weakly harmonic on $M^N$. We prove that if $M$ has the finite-product Liouville property, then every bounded empirically harmonic function on $\Ptwo(M)$ is constant. In particular, the result applies to $M=\R^d$ and to every complete connected Riemannian manifold with nonnegative Ricci curvature. We also derive a finite-particle chain rule for sufficiently regular functionals on $\Ptwo(\R^d)$ and show that the empirical Laplacian is exactly the Hessian trace of a discrete $N$-particle lift. Finally, if $M$ admits a nonconstant bounded harmonic function, then $\Ptwo(M)$ admits a nonconstant bounded empirically harmonic linear statistic.
\end{abstract}

\maketitle

\section{Introduction}

The Liouville property is one of the fundamental rigidity phenomena in elliptic partial differential equations. In its classical Euclidean form, it asserts that every bounded harmonic function on $\R^n$ is constant. On a complete Riemannian manifold $M$, the corresponding statement is not automatic since geometry at infinity influences bounded harmonic functions. Yau proved that every positive harmonic function on a complete Riemannian manifold with nonnegative Ricci curvature is constant \cite{Yau1975}; consequently, under the same curvature assumption, every bounded harmonic function is constant after adding a sufficiently large positive constant. The local gradient estimates of Cheng and Yau \cite{ChengYau1975} and the probabilistic and potential-theoretic theory surveyed by Grigor'yan \cite{Grigoryan1999} form part of the classical background.

Let $(M,g)$ be a complete connected Riemannian manifold, let $o\in M$ be fixed, and let $d_M$ denote the Riemannian distance. The quadratic Wasserstein space is
\[
\Ptwo(M)=\left\{\mu\in\Prob(M):\int_M d_M(x,o)^2\dd\mu(x)<\infty\right\},
\]
endowed with the $2$-Wasserstein distance $\Wtwo$. Otto's formal Riemannian calculus \cite{Otto2001}, the metric theory of gradient flows developed by Ambrosio, Gigli, and Savar\'e \cite{AGS2008}, and the monographs of Villani \cite{Villani2003,Villani2009} make it natural to regard $\Ptwo(M)$ as an infinite-dimensional Riemannian space. Second-order aspects have been developed, among others, by Lott \cite{Lott2008} and Gigli \cite{Gigli2012}. Unlike a finite-dimensional Riemannian manifold, however, $\Ptwo(M)$ has no canonical analogue of Riemannian volume. Consequently, the phrase ``the Laplacian on Wasserstein space'' is ambiguous. Important constructions include the entropic-measure and Wasserstein-diffusion approach of von Renesse and Sturm \cite{vonRenesseSturm2009,Sturm2026} and the partial Wasserstein Laplacian of Chow and Gangbo \cite{ChowGangbo2019}.

Our approach tests a function along every finite-particle empirical map. For $N\geq1$, define
\begin{equation}\label{eq:iota-intro}
\iota_N:M^N\longrightarrow\Ptwo(M),\qquad
\iota_N(x_1,\ldots,x_N)=\frac1N\sum_{i=1}^N\delta_{x_i}.
\end{equation}
The map $\iota_N$ is invariant under coordinate permutations and therefore is not an embedding. Its image is the set of equal-weight empirical measures with $N$ particles. We call a continuous function $u:\Ptwo(M)\to\R$ \emph{empirically harmonic} if $u\circ\iota_N$ is weakly harmonic on $M^N$ for every $N$. The normalized product metric on $M^N$ makes $\iota_N$ $1$-Lipschitz, but injectivity is neither asserted nor needed.

Our main theorem is the following.
\begin{theorem}[Liouville theorem on $\Ptwo(M)$]\label{thm:intro}
Let $M$ be a complete connected Riemannian manifold with the finite-product Liouville property: for every $N\geq1$, every bounded weakly harmonic function on $M^N$ is constant. Then every bounded continuous empirically harmonic function on $\Ptwo(M)$ is constant.
\end{theorem}

The proof is short and robust. Each finite-particle pullback is bounded and harmonic, hence constant by the assumed finite-dimensional Liouville property. The constants agree since all empirical images contain every Dirac measure, and equal-weight empirical measures are dense in $\Ptwo(M)$. The theorem applies to $M=\R^d$ and, more generally, to every complete connected Riemannian manifold with $\Ric\geq0$, since finite products retain completeness and nonnegative Ricci curvature.

The all-$N$ requirement is strong: it imposes exact harmonicity on every particle configuration, not merely an asymptotic equation as $N\to\infty$. To clarify its differential meaning, we derive a chain rule for sufficiently regular functionals on $\Ptwo(\R^d)$. We then introduce a discrete Hilbert lift on an abstract $N$-point probability space and prove that the full Hessian trace of this finite-dimensional lift equals the empirical Laplacian exactly. 

Finally, the conclusion can fail whenever the bounded Liouville property already fails on the base manifold. If $h$ is a nonconstant bounded harmonic function on $M$, then
\[
u_h(\mu)=\int_M h\dd\mu
\]
is a nonconstant bounded continuous empirically harmonic function on $\Ptwo(M)$.

\section{Preliminaries}

Let $(M,g)$ be a complete connected smooth Riemannian manifold. In local coordinates, we use the Laplace--Beltrami operator
\begin{equation}\label{eq:laplace-beltrami}
\Delta f
=\operatorname{div}\nabla f
=\frac{1}{\sqrt{\det g}}\,\partial_i\!\left(\sqrt{\det g}\,g^{ij}\partial_j f\right).
\end{equation}

For $\mu,\nu\in\Ptwo(M)$, the quadratic Wasserstein distance is
\begin{equation}\label{eq:w2}
\Wtwo(\mu,\nu)^2
=
\inf_{\pi\in\Pi(\mu,\nu)}
\int_{M\times M}d_M(x,y)^2\dd\pi(x,y),
\end{equation}
where $\Pi(\mu,\nu)$ is the set of couplings of $\mu$ and $\nu$.

For $N\geq1$, equip $M^N$ with the normalized product metric
\begin{equation}\label{eq:normalized-metric}
g_N=\frac1N\sum_{i=1}^N\pi_i^*g,
\end{equation}
where $\pi_i:M^N\to M$ is the $i$th coordinate projection. The induced distance is
\begin{equation}\label{eq:dN}
d_N(\boldsymbol{x},\boldsymbol{y})^2
=\frac1N\sum_{i=1}^N d_M(x_i,y_i)^2.
\end{equation}
The empirical map
\begin{equation}\label{eq:emp-map}
\iota_N(\boldsymbol{x})=\mu_{\boldsymbol{x}}^N
=\frac1N\sum_{i=1}^N\delta_{x_i}
\end{equation}
is $1$-Lipschitz, since coupling the $i$th atom with the $i$th atom gives
\begin{equation}\label{eq:iota-lip}
\Wtwo\!\left(\mu_{\boldsymbol{x}}^N,\mu_{\boldsymbol{y}}^N\right)^2
\leq\frac1N\sum_{i=1}^N d_M(x_i,y_i)^2
=d_N(\boldsymbol{x},\boldsymbol{y})^2.
\end{equation}

\begin{definition}[Empirical harmonicity]\label{def:EH}
Let $u:\Ptwo(M)\to\R$ be continuous. For $N\geq1$, set
\[
u_N:=u\circ\iota_N,
\qquad
u_N(x_1,\ldots,x_N)=u\!\left(\frac1N\sum_{i=1}^N\delta_{x_i}\right).
\]
We say that $u$ is \emph{empirically harmonic} if, for every $N\geq1$, $u_N$ is weakly harmonic on $(M^N,g_N)$; namely,
\begin{equation}\label{eq:weak-harmonic}
\int_{M^N}u_N\,\Delta_{g_N}\varphi\dd\operatorname{vol}_{g_N}=0
\end{equation}
for every $\varphi\in C_c^\infty(M^N)$.
\end{definition}

\begin{remark}\label{rem:scaling}
Since $g_N$ differs from the usual product metric by a constant factor,
\begin{equation}\label{eq:laplacian-scaling}
\Delta_{g_N}=N\sum_{i=1}^N\Delta_{x_i}.
\end{equation}
Thus empirical harmonicity is equivalent to
$\sum_{i=1}^N\Delta_{x_i}u_N=0$ and does not depend on the normalization. The normalization is used to obtain the Lipschitz estimate \eqref{eq:iota-lip}.
\end{remark}

\begin{definition}[Finite-product Liouville property]\label{def:FPL}
A complete connected Riemannian manifold $M$ has the \emph{finite-product Liouville property} if, for every $N\geq1$, every bounded weakly harmonic function on $M^N$ is constant.
\end{definition}

\begin{lemma}[Density of empirical measures]\label{lem:density}
The set
\[
\bigcup_{N=1}^{\infty}\iota_N(M^N)
\]
is dense in $\Ptwo(M)$ with respect to $\Wtwo$.
\end{lemma}

\begin{proof}
Fix $o\in M$. It is enough to show that finitely supported probability measures with rational weights are dense.

Let $\mu\in\Ptwo(M)$. For $R>0$, define
\[
\mu_R=\mu\!\restriction_{\overline{B_R(o)}}
+\mu\bigl(M\setminus\overline{B_R(o)}\bigr)\delta_o.
\]
By coupling $x\in\overline{B_R(o)}$ to itself and $x\notin\overline{B_R(o)}$ to $o$, we obtain
\[
\Wtwo(\mu,\mu_R)^2
\leq
\int_{M\setminus\overline{B_R(o)}}d_M(x,o)^2\dd\mu(x)
\longrightarrow0
\]
as $R\to\infty$. By the Hopf--Rinow theorem, completeness of $M$ implies that every closed metric ball $\overline{B_R(o)}$ is compact \revcite{Lee2018}. Consequently, each truncation $\mu_R$ is compactly supported; thus the preceding approximation reduces the density argument to compactly supported probability measures.

Assume $\supp\mu\subset K$, where $K\subset M$ is compact. For $\varepsilon>0$, choose a finite Borel partition $K=A_1\cup\cdots\cup A_m$ with $\diam(A_j)\leq\varepsilon$, and choose $y_j\in A_j$. Then
\[
\nu=\sum_{j=1}^m\mu(A_j)\delta_{y_j}
\]
satisfies $\Wtwo(\mu,\nu)\leq\varepsilon$. Choose nonnegative rational numbers $q_j$ with $\sum_jq_j=1$ and with the vector $(q_1,\ldots,q_m)$ arbitrarily close to $(\mu(A_1),\ldots,\mu(A_m))$. If $D=\diam\{y_1,\ldots,y_m\}$, the common-mass coupling gives
\[
\Wtwo\!\left(\sum_j\mu(A_j)\delta_{y_j},\sum_jq_j\delta_{y_j}\right)^2
\leq D^2\,\frac12\sum_{j=1}^m|\mu(A_j)-q_j|.
\]
Thus rationally weighted finitely supported measures are dense. If $q_j=m_j/N$, then
\[
\sum_{j=1}^m q_j\delta_{y_j}
=\frac1N\sum_{j=1}^m\sum_{\ell=1}^{m_j}\delta_{y_j}
\]
is an equal-weight empirical measure. This proves the claim.
\end{proof}

\section{Finite-particle differential calculus on $\Ptwo(\R^d)$}

The Liouville theorem itself requires only Definition~\ref{def:EH} and Lemma~\ref{lem:density}. This section provides a differential interpretation for sufficiently regular functionals on the Euclidean Wasserstein space. All derivatives below are stated on $\Ptwo(\R^d)$.

\subsection{Normalized linear functional derivatives}

Linear functional derivatives are determined only up to additive terms that vanish against signed measures of total mass zero. We use normalized representatives to make the notation unambiguous.

\begin{definition}[First linear functional derivative]\label{def:first-LFD}
A function $f:\Ptwo(\R^d)\to\R$ admits a first linear functional derivative if there is a jointly continuous function
\[
\frac{\delta f}{\delta\mu}:\Ptwo(\R^d)\times\R^d\to\R
\]
such that, for every $\mu,\nu\in\Ptwo(\R^d)$,
\begin{equation}\label{eq:first-variation}
f(\nu)-f(\mu)
=\int_0^1\int_{\R^d}
\frac{\delta f}{\delta\mu}\bigl((1-t)\mu+t\nu\bigr)(x)
\dd(\nu-\mu)(x)\dd t.
\end{equation}
On every $\Wtwo$-bounded set $B\subset\Ptwo(\R^d)$, we assume a quadratic growth bound
\[
\left|\frac{\delta f}{\delta\mu}(\mu)(x)\right|
\leq C_B(1+|x|^2),
\qquad \mu\in B,
\]
and choose the normalized representative satisfying
\begin{equation}\label{eq:first-normalization}
\int_{\R^d}\frac{\delta f}{\delta\mu}(\mu)(x)\dd\mu(x)=0.
\end{equation}
\end{definition}

\begin{definition}[Second linear functional derivative]\label{def:second-LFD}
Assume that $f$ has a first linear functional derivative. A jointly continuous symmetric kernel
\[
\frac{\delta^2 f}{\delta\mu^2}:\Ptwo(\R^d)\times\R^d\times\R^d\to\R
\]
is called a second linear functional derivative if, for every $\mu,\nu\in\Ptwo(\R^d)$ and $\eta=\nu-\mu$,
\begin{align}\label{eq:second-variation}
&f(\nu)-f(\mu)
-\int_{\R^d}\frac{\delta f}{\delta\mu}(\mu)(x)\dd\eta(x)\\
&\quad=
\int_0^1(1-t)
\int_{\R^d}\int_{\R^d}
\frac{\delta^2 f}{\delta\mu^2}\bigl((1-t)\mu+t\nu\bigr)(x,y)
\dd\eta(x)\dd\eta(y)\dd t.\nonumber
\end{align}
On every $\Wtwo$-bounded set $B$, we assume
\[
\left|\frac{\delta^2 f}{\delta\mu^2}(\mu)(x,y)\right|
\leq C_B(1+|x|^2+|y|^2)
\]
and choose a normalized representative satisfying
\begin{equation}\label{eq:second-normalization}
\int_{\R^d}\frac{\delta^2 f}{\delta\mu^2}(\mu)(x,y)\dd\mu(y)=0
\quad\text{for every }x,
\end{equation}
and the analogous identity after interchanging $x$ and $y$.
\end{definition}

\begin{remark}[Gauge freedom]\label{rem:gauge}
Without the normalizations, $\delta f/\delta\mu$ may be changed by a function depending only on $\mu$. Similarly, second-variation kernels may be modified by one-variable gauge terms that do not affect the bilinear expression against signed perturbations of total mass zero. The spatial gradients, Hessians, and mixed spatial derivatives used below are insensitive to the corresponding additive constants.
\end{remark}

\begin{definition}[Second-order regular differentiability; also see \cite{BWYY2026}]\label{def:regularity}
We call $f$ \emph{second-order regularly differentiable} if it admits normalized first and second linear functional derivatives such that\footnote{Here $\operatorname{Sym}^2(\R^d)$ denotes the vector space of symmetric bilinear forms on $\R^d$, equivalently the space of symmetric real $d\times d$ matrices.}
\[
\nabla_x^2\frac{\delta f}{\delta\mu}(\mu)(x)
\in\operatorname{Sym}^2(\R^d),
\qquad
\nabla_x\nabla_y\frac{\delta^2 f}{\delta\mu^2}(\mu)(x,y)
\in\R^{d\times d}
\]
exist and are jointly continuous. We additionally require the compatibility identity
\begin{align}\label{eq:gradient-measure-compatibility}
&\nabla_x\frac{\delta f}{\delta\mu}(\nu)(x)
-\nabla_x\frac{\delta f}{\delta\mu}(\mu)(x)\\
&\quad=\int_0^1\int_{\R^d}
\nabla_x\frac{\delta^2f}{\delta\mu^2}((1-t)\mu+t\nu)(x,y)
\dd(\nu-\mu)(y)\dd t.\nonumber
\end{align}
The relevant spatial derivatives and their operator norms are assumed bounded on every set of the form
\[
B\times K
\quad\text{and}\quad B\times K\times K,
\]
where $B\subset\Ptwo(\R^d)$ is $\Wtwo$-bounded and $K\subset\R^d$ is compact, and to have at most polynomial growth compatible with second moments. The compatibility identity is invariant under the gauge choices described in Remark~\ref{rem:gauge} since spatial gradients remove additive constants.
\end{definition}

\subsection{The particle chain rule}

For $\boldsymbol{x}=(x_1,\ldots,x_N)\in(\R^d)^N$, write
\[
\mu_{\boldsymbol{x}}^N=\frac1N\sum_{i=1}^N\delta_{x_i},
\qquad
f_N(\boldsymbol{x})=f(\mu_{\boldsymbol{x}}^N).
\]
For a smooth kernel $K(x,y)$, set
\[
\Delta_{x,y}K(x,y)
:=\sum_{j=1}^d\partial_{x_j}\partial_{y_j}K(x,y).
\]

\begin{proposition}[Finite-particle chain rule]\label{prop:general-chain-rule}
Let $f:\Ptwo(\R^d)\to\R$ be second-order regularly differentiable. Then $f_N\in C^2((\R^d)^N)$ and, for $1\leq i,k\leq N$ and $1\leq p,q\leq d$,
\begin{align}
\partial_{x_i^p}f_N(\boldsymbol{x})
&=\frac1N\partial_{x^p}\frac{\delta f}{\delta\mu}(\mu_{\boldsymbol{x}}^N)(x_i),\label{eq:first-particle}\\
\partial_{x_i^p x_i^q}^2f_N(\boldsymbol{x})
&=\frac1N\partial_{x^p x^q}^2\frac{\delta f}{\delta\mu}(\mu_{\boldsymbol{x}}^N)(x_i)
+\frac1{N^2}\partial_{x^p y^q}^2\frac{\delta^2f}{\delta\mu^2}(\mu_{\boldsymbol{x}}^N)(x_i,x_i),\label{eq:diag-particle}\\
\partial_{x_i^p x_k^q}^2f_N(\boldsymbol{x})
&=\frac1{N^2}\partial_{x^p y^q}^2\frac{\delta^2f}{\delta\mu^2}(\mu_{\boldsymbol{x}}^N)(x_i,x_k),
\qquad i\neq k.\label{eq:offdiag-particle}
\end{align}
Consequently,
\begin{align}\label{eq:general-empirical-lap}
\sum_{i=1}^N\Delta_{x_i}f_N(\boldsymbol{x})
&=\int_{\R^d}\Delta_x\frac{\delta f}{\delta\mu}(\mu_{\boldsymbol{x}}^N)(y)\dd\mu_{\boldsymbol{x}}^N(y)\\
&\quad+\frac1N\int_{\R^d}\Delta_{x,y}\frac{\delta^2 f}{\delta\mu^2}(\mu_{\boldsymbol{x}}^N)(y,y)\dd\mu_{\boldsymbol{x}}^N(y).\nonumber
\end{align}
Equivalently, by \eqref{eq:laplacian-scaling}, the left-hand side is $N^{-1}\Delta_{g_N}f_N$.
\end{proposition}

\begin{proof}
Fix $i$ and a direction $v\in\R^d$. Move only the $i$th atom by setting
\[
\boldsymbol{x}^{i,v}(t)=(x_1,\ldots,x_i+tv,\ldots,x_N),
\qquad
\mu_t=\mu_{\boldsymbol{x}^{i,v}(t)}^N.
\]
For $h\neq0$, formula \eqref{eq:first-variation} gives
\begin{align*}
\frac{f(\mu_{t+h})-f(\mu_t)}{h}
=\frac1N\int_0^1
\frac{G_{t,h,s}(x_i+(t+h)v)-G_{t,h,s}(x_i+tv)}{h}\dd s,
\end{align*}
where
\[
G_{t,h,s}(z)
=\frac{\delta f}{\delta\mu}\bigl((1-s)\mu_t+s\mu_{t+h}\bigr)(z).
\]
The joint continuity and growth assumptions allow $h\to0$, yielding
\begin{equation}\label{eq:directional-first}
\frac{\dd}{\dd t}f(\mu_t)
=\frac1N\nabla_x\frac{\delta f}{\delta\mu}(\mu_t)(x_i+tv)\cdot v.
\end{equation}
This proves \eqref{eq:first-particle}.

Differentiate the right-hand side of \eqref{eq:directional-first}. The change of its spatial argument gives
\[
\frac1N \nabla_x^2\frac{\delta f}{\delta\mu}(\mu_t)(x_i+tv)[v,v].
\]
For the change of the measure argument, apply the compatibility identity \eqref{eq:gradient-measure-compatibility} to $\mu_t$ and $\mu_{t+h}$. Since
\[
\mu_{t+h}-\mu_t
=\frac1N\bigl(\delta_{x_i+(t+h)v}-\delta_{x_i+tv}\bigr),
\]
division by $h$ and passage to the limit give
\[
\frac1{N^2}\nabla_x\nabla_y\frac{\delta^2f}{\delta\mu^2}(\mu_t)
(x_i+tv,x_i+tv)[v,v].
\]
At $t=0$, polarization in $v$ yields \eqref{eq:diag-particle}. If a distinct particle $x_k$ is moved in a direction $w$, the spatial argument $x_i$ remains fixed and the same compatibility identity gives
\[
\frac1{N^2}\nabla_x\nabla_y\frac{\delta^2f}{\delta\mu^2}(\mu_{\boldsymbol{x}}^N)
(x_i,x_k)[v,w],
\]
which is \eqref{eq:offdiag-particle}. Taking the trace in \eqref{eq:diag-particle}, summing over $i$, and rewriting the sums as integrals against $\mu_{\boldsymbol{x}}^N$ proves \eqref{eq:general-empirical-lap}.
\end{proof}

\subsection{Cylinder functions}

\begin{proposition}[Cylinder calculation]\label{prop:cylinder}
Let $\phi_1,\ldots,\phi_k\in C^2(\R^d)$ and $\Phi\in C^2(\R^k)$, with sufficient boundedness or growth assumptions for the expressions below. Define
\begin{equation}\label{eq:cylinder}
f(\mu)=\Phi\bigl(m_1(\mu),\ldots,m_k(\mu)\bigr),
\qquad
m_a(\mu)=\int_{\R^d}\phi_a\dd\mu.
\end{equation}
Then normalized representatives of its first and second linear functional derivatives are
\begin{align}
\frac{\delta f}{\delta\mu}(\mu)(x)
&=\sum_{a=1}^k\partial_a\Phi(m(\mu))\bigl(\phi_a(x)-m_a(\mu)\bigr),\label{eq:cyl-first}\\
\frac{\delta^2 f}{\delta\mu^2}(\mu)(x,y)
&=\sum_{a,b=1}^k\partial_{ab}\Phi(m(\mu))
\bigl(\phi_a(x)-m_a(\mu)\bigr)
\bigl(\phi_b(y)-m_b(\mu)\bigr).\label{eq:cyl-second}
\end{align}
For $m_a(\boldsymbol{x})=N^{-1}\sum_i\phi_a(x_i)$, one has
\begin{align}\label{eq:cylinder-lap}
\frac1N\Delta_{g_N}f_N(\boldsymbol{x})
&=\sum_{a=1}^k\partial_a\Phi(m(\boldsymbol{x}))
\int_{\R^d}\Delta\phi_a\dd\mu_{\boldsymbol{x}}^N\\
&\quad+\frac1N\sum_{a,b=1}^k\partial_{ab}\Phi(m(\boldsymbol{x}))
\int_{\R^d}\ip{\nabla\phi_a}{\nabla\phi_b}\dd\mu_{\boldsymbol{x}}^N.\nonumber
\end{align}
\end{proposition}

\begin{proof}
The first-variation identity follows by applying the ordinary fundamental theorem of calculus to $t\mapsto\Phi(m((1-t)\mu+t\nu))$. Formula \eqref{eq:cyl-first} is centered and therefore satisfies \eqref{eq:first-normalization}. Applying the second-order Taylor formula to the same finite-dimensional curve gives \eqref{eq:cyl-second}; the product centering verifies both normalizations in \eqref{eq:second-normalization}.

The centering terms are constant in the spatial variables. Hence
\[
\Delta_x\frac{\delta f}{\delta\mu}(\mu)(x)
=\sum_a\partial_a\Phi(m(\mu))\Delta\phi_a(x)
\]
and
\[
\Delta_{x,y}\frac{\delta^2f}{\delta\mu^2}(\mu)(x,y)
=\sum_{a,b}\partial_{ab}\Phi(m(\mu))
\ip{\nabla\phi_a(x)}{\nabla\phi_b(y)}.
\]
Substitution into Proposition~\ref{prop:general-chain-rule} proves \eqref{eq:cylinder-lap}.
\end{proof}

\begin{remark}[The finite-particle fluctuation term]\label{rem:fluctuation}
The first term in \eqref{eq:cylinder-lap} is the lift of the base Laplacian acting on linear statistics. The second is a diagonal second-variation term of order $1/N$, generated by moving one atom at a time. For example, if
\[
f(\mu)=\Phi\!\left(\int\phi\dd\mu\right),
\]
then
\[
\frac1N\Delta_{g_N}f_N
=\Phi'(m)\int\Delta\phi\dd\mu_{\boldsymbol{x}}^N
+\frac1N\Phi''(m)\int|\nabla\phi|^2\dd\mu_{\boldsymbol{x}}^N.
\]
Thus even when $\phi$ is harmonic, a nonlinear choice of $\Phi$ generally fails to be empirically harmonic. This illustrates the strength of exact harmonicity for every $N$.
\end{remark}

\subsection{Necessary identities at Dirac masses}

\begin{proposition}[A necessary Dirac-mass test]\label{prop:dirac-test}
Let $f$ be second-order regularly differentiable and empirically harmonic on $\Ptwo(\R^d)$. Then, for every $y\in\R^d$,
\begin{align}
\Delta_x\frac{\delta f}{\delta\mu}(\delta_y)(y)&=0,\label{eq:dirac-first}\\
\Delta_{x,z}\frac{\delta^2 f}{\delta\mu^2}(\delta_y)(y,y)&=0.\label{eq:dirac-second}
\end{align}
These identities are necessary but are not asserted to be sufficient for empirical harmonicity.
\end{proposition}

\begin{proof}
Set $x_1=\cdots=x_N=y$ in \eqref{eq:general-empirical-lap}. Empirical harmonicity gives
\[
A(y)+\frac1N B(y)=0
\qquad\text{for every }N\geq1,
\]
where $A$ and $B$ are the left-hand sides of \eqref{eq:dirac-first} and \eqref{eq:dirac-second}. Taking, for example, $N=1$ and $N=2$ yields $B(y)=0$ and then $A(y)=0$.
\end{proof}

\section{A discrete Hilbert lift and its Hessian trace}

Let $A_N=\{a_1,\ldots,a_N\}$ be an abstract $N$-point set with the uniform probability measure
\[
m_0^N=\frac1N\sum_{i=1}^N\delta_{a_i}.
\]
Set
\[
H_N=L^2(m_0^N;\R^d),
\qquad
\ip{Y}{Z}_{H_N}=\frac1N\sum_{i=1}^N Y(a_i)\cdot Z(a_i).
\]
After identifying $X\in H_N$ with $(X(a_1),\ldots,X(a_N))\in(\R^d)^N$, define
\begin{equation}\label{eq:discrete-lift}
F_N(X)=f(X_\#m_0^N)
=f\!\left(\frac1N\sum_{i=1}^N\delta_{X(a_i)}\right).
\end{equation}

\begin{proposition}[Exact discrete trace identity]\label{prop:discrete-trace}
Assume that $f_N$ is twice differentiable. For $1\leq i\leq N$ and $1\leq j\leq d$, define
\[
E_{i,j}=\sqrt{N}\,\mathbf{1}_{\{a_i\}}\e_j\in H_N.
\]
Then $(E_{i,j})_{i,j}$ is a complete orthonormal basis of $H_N$, and
\begin{equation}\label{eq:exact-trace}
\Tr_{H_N}D^2F_N(X)
=\sum_{i=1}^N\sum_{j=1}^dD^2F_N(X)[E_{i,j},E_{i,j}]
=\Delta_{g_N}f_N(\boldsymbol{x}),
\end{equation}
where $\boldsymbol{x}=(X(a_1),\ldots,X(a_N))$.
\end{proposition}

\begin{proof}
The normalization of the inner product gives $\|E_{i,j}\|_{H_N}=1$, and the family is plainly orthogonal and complete. The curve $X+tE_{i,j}$ moves only the $i$th particle, with velocity $\sqrt{N}\e_j$. Therefore
\[
D^2F_N(X)[E_{i,j},E_{i,j}]
=N\,\partial_{x_i^j x_i^j}^2f_N(\boldsymbol{x}).
\]
Summing over $i$ and $j$ gives
\[
\Tr_{H_N}D^2F_N(X)
=N\sum_{i=1}^N\Delta_{x_i}f_N(\boldsymbol{x})
=\Delta_{g_N}f_N(\boldsymbol{x})
\]
by \eqref{eq:laplacian-scaling}.
\end{proof}

\begin{corollary}\label{cor:trace-harmonic}
If every pullback $f_N$ is $C^2$, then $f$ is empirically harmonic if and only if, for every $N\geq1$, its discrete lift $F_N$ satisfies
\[
\Tr_{H_N}D^2F_N=0.
\]
\end{corollary}

Let $\e_j\in H_N$ also denote the constant vector field $a\mapsto\e_j$. These $d$ vectors form an orthonormal basis only of the $d$-dimensional subspace of common translations. Define
\[
\Delta_{\mathrm{trans}}F_N(X)
:=\sum_{j=1}^dD^2F_N(X)[\e_j,\e_j].
\]
For a second-order regularly differentiable $f$, formulas \eqref{eq:diag-particle}--\eqref{eq:offdiag-particle} give
\begin{align}\label{eq:translation-trace}
\Delta_{\mathrm{trans}}F_N(X)
&=\int_{\R^d}\Delta_x\frac{\delta f}{\delta\mu}(\mu)(x)\dd\mu(x)\\
&\quad+\int_{\R^d}\int_{\R^d}
\Delta_{x,y}\frac{\delta^2f}{\delta\mu^2}(\mu)(x,y)
\dd\mu(x)\dd\mu(y),\nonumber
\end{align}
where $\mu=X_\#m_0^N$. The double integral contains all particle pairs since a common translation moves every atom simultaneously. In contrast, the empirical Laplacian contains only the diagonal second-variation contribution, with coefficient $1/N$, as shown in \eqref{eq:general-empirical-lap}.

\begin{corollary}[Compatibility at Dirac masses]\label{cor:dirac-compatibility}
If $f$ is second-order regularly differentiable and empirically harmonic, then
\[
\Delta_{\mathrm{trans}}F_N(X_y)=0
\]
for every constant map $X_y(a_i)=y$.
\end{corollary}

\begin{proof}
At $\mu=\delta_y$, formula \eqref{eq:translation-trace} is the sum of the two quantities in \eqref{eq:dirac-first} and \eqref{eq:dirac-second}, both of which vanish by Proposition~\ref{prop:dirac-test}.
\end{proof}

\section{Main Liouville theorem}

\begin{theorem}[Liouville theorem for empirically harmonic functions]\label{thm:main}
Let $M$ be a complete connected Riemannian manifold with the finite-product Liouville property. If
\[
u\in C(\Ptwo(M))\cap L^\infty(\Ptwo(M))
\]
is empirically harmonic, then $u$ is constant.
\end{theorem}

\begin{proof}
Fix $N\geq1$. By empirical harmonicity,
\[
u_N(\boldsymbol{x})
=u\!\left(\frac1N\sum_{i=1}^N\delta_{x_i}\right)
\]
is weakly harmonic on $M^N$. Since $u$ is bounded on $\Ptwo(M)$, every pullback $u_N$ is bounded on $M^N$. By the finite-product Liouville property, $u_N$ is constant; denote its value by $c_N$.

For any fixed $y\in M$,
\[
c_N=u_N(y,\ldots,y)=u(\delta_y)=u_1(y)=c_1.
\]
Thus $u=c_1$ on every empirical image $\iota_N(M^N)$. By Lemma~\ref{lem:density}, the union of these images is dense in $\Ptwo(M)$. Since $u$ is continuous with respect to $\Wtwo$, it follows that $u\equiv c_1$ on all of $\Ptwo(M)$.
\end{proof}

\begin{corollary}[Euclidean base]\label{cor:euclidean}
Let $M=\R^d$. Then every bounded continuous empirically harmonic function on $\Ptwo(\R^d)$ is constant.
\end{corollary}

\begin{proof}
For every $N\geq1$, $M^N=\R^{dN}$. The classical Liouville theorem for bounded harmonic functions on Euclidean space gives the finite-product Liouville property, so Theorem~\ref{thm:main} applies.
\end{proof}

\begin{corollary}[Nonnegative Ricci curvature]\label{cor:ricci}
Let $M$ be a complete connected Riemannian manifold with $\Ric_M\geq0$. Then every bounded continuous empirically harmonic function on $\Ptwo(M)$ is constant.
\end{corollary}

\begin{proof}
For every $N\geq1$, the product manifold $M^N$ is complete and has nonnegative Ricci curvature. Multiplying the product metric by the positive constant $1/N$ preserves harmonicity and the sign of the Ricci tensor. Let $v$ be a bounded weakly harmonic function on $M^N$. By elliptic regularity, $v$ has a smooth harmonic representative. Choose $A>\|v\|_{L^\infty}+1$. Then $v+A$ is a positive harmonic function. Yau's Liouville theorem implies that $v+A$, and hence $v$, is constant. Thus $M$ has the finite-product Liouville property, and Theorem~\ref{thm:main} applies.
\end{proof}

\section{A partial converse and examples}

The next proposition shows that the conclusion of Theorem~\ref{thm:main} can fail whenever the bounded Liouville property already fails on the base manifold. 

\begin{proposition}[Counterexample from the base manifold]\label{prop:partial-converse}
Assume that $M$ admits a nonconstant bounded harmonic function $h:M\to\R$. Then
\[
u_h:\Ptwo(M)\longrightarrow\R,
\qquad
u_h(\mu)=\int_M h\dd\mu,
\]
is a nonconstant bounded continuous empirically harmonic function.
\end{proposition}

\begin{proof}
Since $h$ is bounded, $u_h$ is bounded. Since $\Wtwo$-convergence implies weak convergence of probability measures and $h$ is bounded and continuous, $u_h$ is continuous on $\Ptwo(M)$.

For $\boldsymbol{x}=(x_1,\ldots,x_N)$,
\[
(u_h)_N(\boldsymbol{x})=\frac1N\sum_{i=1}^Nh(x_i).
\]
Therefore
\[
\sum_{i=1}^N\Delta_{x_i}(u_h)_N
=\frac1N\sum_{i=1}^N\Delta h(x_i)=0.
\]
Thus $u_h$ is empirically harmonic. Finally, $u_h(\delta_x)=h(x)$, so $u_h$ is nonconstant whenever $h$ is nonconstant.
\end{proof}

\begin{example}[A nonlinear obstruction]\label{ex:nonlinear}
Let $M=\R^d$, let $\phi\in C^2(\R^d)$ be harmonic, and consider
\[
f(\mu)=\Phi\!\left(\int\phi\dd\mu\right).
\]
By Remark~\ref{rem:fluctuation},
\[
\frac1N\Delta_{g_N}f_N
=\frac1N\Phi''(m)\int|\nabla\phi|^2\dd\mu_{\boldsymbol{x}}^N.
\]
Hence nonlinear post-composition generally destroys empirical harmonicity. For instance, if there is a point $x$ with $\nabla\phi(x)\neq0$ and $\Phi''(\phi(x))\neq0$, then the $N=1$ pullback is not harmonic at $x$. This contrasts with the linear statistic in Proposition~\ref{prop:partial-converse}.
\end{example}

\section{Concluding remarks}

Theorem~\ref{thm:main} is a Liouville theorem for a finite-particle trace notion of harmonicity on $\Ptwo(M)$. Its mechanism is transparent:
\[
\Ptwo(M)=\overline{\bigcup_{N\geq1}\iota_N(M^N)}^{\,\Wtwo},
\]
and exact finite-dimensional Liouville rigidity on each particle space propagates to the Wasserstein space by continuity.

The discrete lift in Section~4 gives a precise second-order interpretation. For each $N$, the empirical Laplacian is the full Hessian trace on the finite-dimensional Hilbert space $H_N=L^2(m_0^N;\R^d)$. By contrast, tracing only along constant vector fields measures common translations and produces all-pair terms; it is a partial directional trace, not an infinite-dimensional Hilbert-space trace. These two operations should therefore be kept conceptually distinct.

Several questions remain open within this framework. One may study subharmonic functions, polynomial-growth solutions, $p$-harmonic analogues, or equations imposed only for selected particle numbers. It would also be useful to compare exact all-$N$ harmonicity with asymptotic conditions as $N\to\infty$ and with harmonicity associated with specific Dirichlet forms on $\Ptwo(M)$. Those problems require additional analytic input beyond the boundedness and density argument used here.

% \section*{Acknowledgment}
% Hongwei Yuan is supported by the National Natural Science Foundation of China, Key Program (Grant No.~12531010).

% \section*{Conflict of interest statement}
% The authors have no relevant financial or non-financial interests to disclose.

% \section*{Data availability statement}
% Data sharing is not applicable to this manuscript because no datasets were generated or analyzed. The work is entirely theoretical.


\begin{thebibliography}{99}

\bibitem[AGS08]{AGS2008}
L.~Ambrosio, N.~Gigli, and G.~Savar\'e,
\emph{Gradient Flows in Metric Spaces and in the Space of Probability Measures},
2nd ed., Lectures in Mathematics ETH Z\"urich, Birkh\"auser, Basel, 2008,
\doi{10.1007/978-3-7643-8722-8}.

\bibitem[BWYY26]{BWYY2026}
A.~Bensoussan, T.~K.~Wong, S.~C.~P.~Yam, and H.~Yuan,
\emph{A Theory of First Order Mean Field Type Control Problems and their Equations},
Journal of the European Mathematical Society (2026), published online first,
\doi{10.4171/JEMS/1781}.

\bibitem[CG19]{ChowGangbo2019}
Y.~T. Chow and W.~Gangbo,
\emph{A partial Laplacian as an infinitesimal generator on the Wasserstein space},
Journal of Differential Equations \textbf{267} (2019), no.~10, 6065--6117,
\doi{10.1016/j.jde.2019.06.012}.

\bibitem[CY75]{ChengYau1975}
S.~Y. Cheng and S.~T. Yau,
\emph{Differential equations on Riemannian manifolds and their geometric applications},
Communications on Pure and Applied Mathematics \textbf{28} (1975), no.~3, 333--354,
\doi{10.1002/cpa.3160280303}.

\bibitem[G12]{Gigli2012}
N.~Gigli,
\emph{Second order analysis on $(\mathcal P_2(M),W_2)$},
Memoirs of the American Mathematical Society \textbf{216} (2012), no.~1018, xii+154 pp.,
\doi{10.1090/S0065-9266-2011-00619-2}.

\bibitem[Gr99]{Grigoryan1999}
A.~Grigor'yan,
\emph{Analytic and geometric background of recurrence and non-explosion of the Brownian motion on Riemannian manifolds},
Bulletin of the American Mathematical Society \textbf{36} (1999), no.~2, 135--249,
\doi{10.1090/S0273-0979-99-00776-4}.

\bibitem[Lee18]{Lee2018}
J.~M. Lee,
\emph{Introduction to Riemannian Manifolds},
2nd ed., Graduate Texts in Mathematics, vol.~176, Springer, Cham, 2018,
\doi{10.1007/978-3-319-91755-9}.


\bibitem[Lo08]{Lott2008}
J.~Lott,
\emph{Some geometric calculations on Wasserstein space},
Communications in Mathematical Physics \textbf{277} (2008), no.~2, 423--437,
\doi{10.1007/s00220-007-0367-3}.

\bibitem[Ot01]{Otto2001}
F.~Otto,
\emph{The geometry of dissipative evolution equations: the porous medium equation},
Communications in Partial Differential Equations \textbf{26} (2001), no.~1--2, 101--174,
\doi{10.1081/PDE-100002243}.

\bibitem[St26]{Sturm2026}
K.-T. Sturm,
\emph{Wasserstein diffusion on multidimensional spaces},
The Annals of Probability \textbf{54} (2026), no.~2, 610--643,
\doi{10.1214/25-AOP1774}.

\bibitem[vRS09]{vonRenesseSturm2009}
M.-K. von Renesse and K.-T. Sturm,
\emph{Entropic measure and Wasserstein diffusion},
The Annals of Probability \textbf{37} (2009), no.~3, 1114--1191,
\doi{10.1214/08-AOP430}.

\bibitem[Vi03]{Villani2003}
C.~Villani,
\emph{Topics in Optimal Transportation},
Graduate Studies in Mathematics, vol.~58, American Mathematical Society, Providence, RI, 2003,
\doi{10.1090/gsm/058}.

\bibitem[Vi09]{Villani2009}
C.~Villani,
\emph{Optimal Transport: Old and New},
Grundlehren der mathematischen Wissenschaften, vol.~338, Springer, Berlin, 2009,
\doi{10.1007/978-3-540-71050-9}.

\bibitem[Yau75]{Yau1975}
S.-T. Yau,
\emph{Harmonic functions on complete Riemannian manifolds},
Communications on Pure and Applied Mathematics \textbf{28} (1975), no.~2, 201--228,
\doi{10.1002/cpa.3160280203}.

\end{thebibliography}
\end{document}